\documentclass[reqno,11pt]{article}
\usepackage[T1]{fontenc}
\usepackage{lmodern}
\usepackage{amsmath,amssymb,amsthm,mathtools}
\usepackage[margin=1.12in]{geometry}
\usepackage[colorlinks=true,citecolor=blue,linkcolor=blue,urlcolor=blue]{hyperref}
\usepackage{microtype}
\usepackage{tikz}
\usepackage{authblk}

\newcommand{\R}{\mathbb R}
\newcommand{\Z}{\mathbb Z}
\newcommand{\N}{\mathbb N}
\newcommand{\ind}{\mathbf 1}
\newcommand{\prim}{\mathrm{prim}}
\newcommand{\Area}{\operatorname{Area}}
\newcommand{\Var}{\operatorname{Var}}
\newcommand{\E}{\mathbb E}
\newcommand{\TV}{\mathrm{TV}}
\newtheorem{theorem}{Theorem}[section]
\newtheorem{lemma}[theorem]{Lemma}
\newtheorem{proposition}[theorem]{Proposition}

\newtheorem{remark}[theorem]{Remark}
\title{Mixing of the BCZ map}
\author[1]{Albert Artiles}
\author[2]{Andrew Best}
\author[3]{Matteo Bordignon}
\affil[1]{Yau Mathematical Sciences Center, Tsinghua University}
\affil[2]{Beijing Institute of Mathematical Sciences and Applications}
\affil[3]{University of Turku}
\date{\today}
\begin{document}
\maketitle

\begin{abstract}
We prove that the Boca--Cobeli--Zaharescu (BCZ) map is strongly mixing with respect to normalized Lebesgue measure on the BCZ triangle.
\end{abstract}

\section{Introduction}
The Boca--Cobeli--Zaharescu (BCZ) map is the map $T:\Omega\rightarrow\Omega$ given by

\begin{equation*}
T(a,b) \coloneq \left(b,\left\lfloor\frac{1+a}{b}\right\rfloor b-a\right),
\end{equation*}
where $\Omega \coloneq \{(a,b)\in (0,1]^2: a+b>1\}$ is the BCZ triangle. This map was first used in \cite{BCZHall} to study the distribution of the Farey sequence and confirm a conjecture by R. Hall \cite{Hall1994}.

While studying the dynamics of the horocycle flow on the space of unimodular lattices, Athreya and Cheung \cite{AC} rediscovered the BCZ map. They identified it as the return map of the horocycle flow to the cross-section of unimodular lattices on the plane with a nonzero horizontal vector of length at most 1. This correspondence allowed the BCZ map to be studied using the theory of homogeneous dynamics. In particular, this new perspective allowed Athreya and Cheung to answer two of the questions about the BCZ map posed in \cite{BZ} by Boca and Zaharescu. They showed that $T$ is an ergodic map with respect to the uniform probability measure on $\Omega$ and that $T$ has measure-theoretic entropy zero. Later, during a 2012 workshop, Athreya, Boca, Chaika, Cheung, Chevallier, Leli\`evre, Tseng, and Quas wrote a report in which they asked whether the BCZ map is (strongly) mixing and whether it is weakly mixing \cite{Banff}.

The question of weak mixing was resolved in 2024 by Cheung and Quas \cite{CQ}, but the question of strong mixing remained open. Our main theorem is the following:

\begin{theorem}\label{thm: BCZ is mixing}
    The BCZ map is strongly mixing.
\end{theorem}

The main tool to prove this theorem is a bound on the discrepancy of return times of the BCZ map to an increasing sequence of subsets of $\Omega$, similar to the proof of Cheung and Quas for weak mixing. Their argument uses a positive lower bound on the probability that the $n$th induced return occurs at ordinary time $n+1$. Our argument bounds the probability of zero discrepancy and uses this bound to identify the limits of graph measures.

We now explain the key bound. Let $c\in (0,1]$ and define
\[
A_c \coloneq \{(x,y)\in\Omega: x\leq c\}.
\]
Let $\N$ be the set of positive integers and $\N_0=\N\cup\{0\}$. Denote by $T_c : A_c \to A_c$ the return map to $A_c$ induced by $T$, and write $r_c : A_c \to \N$ for the return time to $A_c$ under $T$, namely
\[ r_c(x) \coloneq \min\{ j \in \N : T^jx \in A_c\}, \]
so $T_c = T^{r_c}|_{A_c}$. Let $m$ and $m_c$ be, respectively, the uniform probability measures on $\Omega$ and $A_c$, so $dm=2\,da\,db$. Then $(\Omega,T,m)$ and $(A_c,T_c,m_c)$ are measure-preserving systems. Moreover, the map $\phi_c: \Omega \rightarrow A_c$ defined as
\begin{equation}\label{eq:phi}
 \phi_c(a,b) \coloneq \left(ca,\,cb+ca\left\lfloor\frac{1-cb}{ca}\right\rfloor\right)
\end{equation}
conjugates $T$ to $T_c$; explicitly, $T_c\circ\phi_c = \phi_c\circ T$. This conjugacy was known to Athreya and Cheung and was a key observation in the proof that the BCZ map is weakly mixing by Cheung and Quas \cite{CQ}.

It will be useful to introduce a parameter $u>0$ that we will later drive to infinity. For $n > u$, define $c_n = c_n(u) \coloneq 1-\frac{u}{n}$. Given such $u$ and $n$, define the \textbf{return times discrepancy} $D_n:\Omega\rightarrow \N_0$ by
\begin{equation}\label{definition of D_n}
 D_n(x) \coloneq \sum_{j=0}^{n-1}\bigl(r_{c_n}(T_{c_n}^j\phi_{c_n}x)-1\bigr).
\end{equation}
Thus $D_n(x)$ counts the number of visits (under $T$) that $\phi_{c_n}x$ makes to $\Omega\setminus A_{c_n}$ before the $n$th return (under $T$) of $\phi_{c_n}x$ to $A_{c_n}$. Note that the $n$th return of $\phi_{c_n}x$ to $A_{c_n}$ occurs at time $n+D_n(x)$. See Figure~\ref{Fig: Visualization of D_n} for an example.

\begin{figure}[t]
\centering
\begin{tikzpicture}[x=6cm,y=6cm,font=\small]
  \fill[gray!7] (0,1)--(1,1)--(1,0)--cycle;
  \draw[->] (0,0)--(1.12,0) node[right] {$a$};
  \draw[->] (0,0)--(0,1.12) node[above] {$b$};

  \draw[thick] (0,1)--(1,1)--(1,0);
  \draw[thick,dashed] (0,1)--(1,0);

  \foreach \t in {0.2,0.4,0.6,0.8,1} {
    \draw (\t,0.005)--(\t,-0.005) node[below] {$\t$};
    \draw (0.005,\t)--(-0.005,\t) node[left] {$\t$};
  }
  \node[below left] at (0,0) {$0$};
  \node at (0.52,1.05) {$A_{c}$};

  \draw[thick]
    ({7/8},{1/8})--({7/8},1)
    node[above] {$c$};

  \foreach \p in {(0,1),(1,0)}
    \draw[fill=white] \p circle[radius=0.8pt];

  \fill[blue]
    ({133/160},{21/80}) circle[radius=2pt]
    node[left,text=black] {$y$};

  \fill[blue] ({21/80},{119/160}) circle[radius=2pt]
    node[right,text=black] {$Ty$};

  \fill[blue] ({119/160},{77/160}) circle[radius=2pt]
    node[left,text=black] {$T^2y$};

  \fill[blue] ({77/160},{7/10}) circle[radius=2pt]
    node[right,text=black] {$T^3y$};

  \fill[blue] ({7/10},{147/160}) circle[radius=2pt]
    node[left,text=black] {$T^4y$};

  \fill[red] ({147/160},{7/32}) circle[radius=2pt]
     node[below,text=black]
      {$ \;\, T^5y$};

  \fill[blue] ({7/32},{133/160}) circle[radius=2pt]
    node[right,text=black] {$T^6y$};

  \fill[blue] ({133/160},{49/80}) circle[radius=2pt]
    node[left,text=black] {$T^7y$};
\end{tikzpicture}
\caption{Take $u=3/4$, $n=6$, $c=c_6=7/8$, and $x=(0.95,0.3)$, which does not appear in the figure. The red and blue points are the first 8 elements in the orbit of $y \coloneq \phi_{c}(x)$ under $T$, each colored red if and only if it does not belong to the triangle $A_{c}$. Since $T_{c}^6y=T^7y$, we have $D_6(x)=1$, the number of red points before the sixth induced return.}
\label{Fig: Visualization of D_n}
\end{figure}
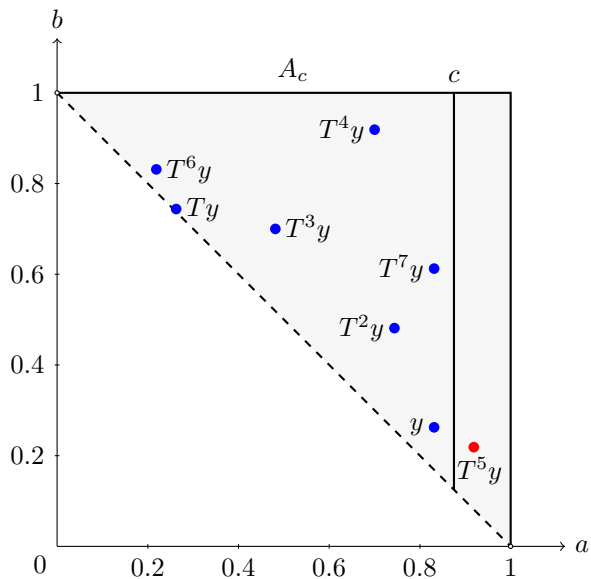

The parameter $u$ determines the horizontal width $u/n$ of $\Omega\setminus A_{c_n}$. The following theorem bounds the asymptotic probability of zero return times discrepancy as $n\to\infty$ with $u>0$ fixed. From this theorem, we will deduce Theorem~\ref{thm: BCZ is mixing}.

\begin{theorem}\label{thm:defect}
For every $u>0$,
\[
 \limsup_{n\to\infty}m\left(\{x\in\Omega:D_n(x)=0\}\right)\le\frac{1}{1+2u}.
\]
\end{theorem}

\subsection{Strategy of the proofs}

In the proof of Theorem~\ref{thm: BCZ is mixing} (see Section~\ref{sec:graph}), we fix an arbitrary subsequential limit of graph measures. For each $u>0$, we decompose this limit according to the return times discrepancy. The weak mixing of $T$ allows us to show that every component corresponding to positive discrepancy is a scalar multiple of $m\otimes m$. Theorem~\ref{thm:defect} bounds the mass of the remaining zero-discrepancy component by $1/(1+2u)$. Although the decomposition may depend on $u$, the limit of the graph measures remains fixed. Letting $u\to\infty$, we conclude that this limit equals $m\otimes m$. Since the subsequential limit was arbitrary, the graph measures converge to $m\otimes m$, proving strong mixing.

Theorem~\ref{thm:defect} is proved by comparing the law of $D_n$ with that of a primitive lattice-point count in a thin rectangle under Haar measure, then using the count's mean and variance to bound the probability that the rectangle is empty. The proof combines the propositions stated at the beginning of Sections~\ref{sec: statistics in thin boxes}~and~\ref{sec: D_n Statistics}.

\section{Preliminaries}\label{sec:Preliminaries}

\subsection{Weak and strong mixing}
Let $(Y,\mathcal{C},\nu,f)$ be a probability measure-preserving system. The system is \textbf{weakly mixing} if the product system $f \times f$ is ergodic. The system is \textbf{(strongly) mixing} if, for every $A,B\in\mathcal{C}$,
\[
    \lim_{n\to\infty}
    \nu\bigl(A\cap f^{-n}(B)\bigr) = \nu(A)\nu(B).
\]
Strong mixing implies weak mixing. For further background,
see \cite{Walters}.

\subsection{Unimodular lattices}\label{Subsec: Unimodular lattices}
A unimodular lattice in $\R^n$ is a discrete subgroup $H$ of rank $n$ such that the quotient $\R^n/H$ has volume one. In this paper, we consider only unimodular lattices in $\R^2$, and we denote the space of these lattices by $X$. Every lattice $H\in X$ can be written as $H=g\Z^2$ for some $g\in \mathrm{SL}(2,\R)$. Thus the natural action of $\mathrm{SL}(2,\R)$ on $X$, given by $g\cdot H=gH$, is transitive. Moreover, the stabilizer of the standard lattice $\Z^2$ is precisely $\mathrm{SL}(2,\Z)$. We therefore identify $X$ with the homogeneous space
\[
    X=\mathrm{SL}(2,\R)/\mathrm{SL}(2,\Z),
\]
where the coset $g\,\mathrm{SL}(2,\Z)$ corresponds to the lattice $g\Z^2$. We denote by $\mu$ the unique $SL(2,\R)$-invariant probability measure on $X$. Although $X$ is not a group, it is a convention to call $\mu$ the Haar measure of $X$.

The horocycle flow is the flow on $X$ given by $(s,g\Z^2)\mapsto (h_sg)\Z^2$, where

 \[
 h_s\coloneq \begin{pmatrix}1&0\\-s&1\end{pmatrix}.
\]

Since $h_s\in SL(2,\R)$, it follows that $h_s$ preserves the measure $\mu$.

\subsection{Horizontally short lattices}

For $h > 0$, a unimodular lattice is called \textbf{$h$-horizontally short} if it has a nonzero horizontal vector of length at most $h$, and if $h=1$, we simply say it is \textbf{horizontally short}.

For $(a,b)\in \Omega$, we write
\[ p_{a,b}\coloneq \begin{pmatrix}a&b\\0&a^{-1}\end{pmatrix} .
\]

It follows from the arguments in \cite[Lemma 2.1 and Section 2.3]{AC} that, for every $h\in(0,1]$, each $h$-horizontally short lattice can be written in the form $p_{a,b}\Z^2$ for a unique pair $(a,b)\in A_h$. Therefore, since $A_1 = \Omega$, we identify $\Omega$ and the space of horizontally short lattices via the map $(a,b)\mapsto p_{a,b}\Z^2$. Under these coordinates, Athreya and Cheung showed that the BCZ map is the return map of the horocycle flow to $\Omega$. 

\subsection{The horocycle flow}\label{subsec: The horocycle flow}

Let $H$ be a lattice. A vector $v$ is \textbf{primitive} in $H$ if $v\in H$ and $tv\notin H$ for all $t\in(0,1)$. We write $H_{\prim}$ for the set of primitive vectors in $H$. For $x=(a,b)\in\Omega$, put $H_x\coloneq p_{a,b}\Z^2$ and $\Lambda_x\coloneq(H_x)_{\prim}$. The vectors in $H_x$ have coordinates $(a\ell+bq,q/a)$, where $(\ell,q)\in \Z^2$. Primitivity is then equivalent to $\gcd(\ell,q)=1$.

A vector $(v_1,v_2)$ with $v_1>0$ becomes horizontal under $h_s$ at time $s=v_2/v_1$. Athreya and Cheung \cite[Lemma 2.2]{AC} showed that the first-return time of the section under $h_s$ is given by the roof function
\begin{equation}\label{Eq: return time horocycle}
    R(a,b) \coloneq \frac1{ab}.
\end{equation}

\noindent In particular, $\mu$-almost every unimodular lattice can be represented uniquely as $h_sH_{(a,b)}$ with $0\leq s<R(a,b)$. The mean return time is
\[\bar\rho \coloneq \int_\Omega R\,dm=2\zeta(2)=\frac{\pi^2}{3}.\]

\noindent The suspension formula (see \cite[Theorem 1.1 and Section 3.1]{AC}) states that, for every bounded Borel function $F:X\to\R$,
\begin{equation*}
 \int_X F\,d\mu=\frac1{\bar\rho}\int_\Omega\int_0^{R(x)}F(h_sH_x)\,ds\,dm(x).
\end{equation*}
Thus $\bar\rho$ is the mass of the unnormalized suspension measure, and in these coordinates \[d\mu=(\bar\rho)^{-1}ds\,dm.\]

For $x\in \Omega$, we define $t_n(x)$ to be the $n$th return time of $x$ to $\Omega$ under the horocycle flow. By \eqref{Eq: return time horocycle}, we have
\begin{equation}\label{definition of t_n} t_n(x)=\sum_{j=0}^{n-1}R(T^jx).\end{equation}

\subsection{Infinitesimal self-similarity}
We collect here the necessary identities that will be used in both the graph-measure argument of Section~\ref{sec:graph} and the lattice-count comparison of Section~\ref{sec: D_n Statistics}. Let $c \in (0,1]$. Then $m(A_c)=c^2$ and $m_c=c^{-2}m|_{A_c}$, so Kac's lemma implies
\begin{equation}\label{eq:kac}
 \int_{A_c}(r_c-1)\,dm_c=c^{-2}-1.
\end{equation}

 The BCZ map is isomorphic to the return map for the horocycle flow to $A_c$ for every $c\in (0,1]$. This property of the BCZ map is called \textbf{infinitesimal self-similarity} and was already known to Athreya and Cheung \cite{AC}. We now explain this relation in more detail.
Let $d_c=\operatorname{diag}(c,c^{-1})$. Write $d_cH_{(a,b)}$ in the coordinates of the BCZ cross-section by adding an integral multiple of its short horizontal basis vector to the other basis vector. This gives, for all $(a,b) \in \Omega$,
\begin{equation*}
 H_{\phi_c(a,b)}=d_cH_{(a,b)}.
\end{equation*}
The same holds for primitive vectors: $\Lambda_{\phi_c(a,b)}=d_c\Lambda_{(a,b)}$. The change of time $h_{s/c^2}d_c=d_ch_s$ identifies successive returns to $\Omega$ with successive returns to $A_c$, and hence $T_c \circ \phi_c=\phi_c \circ T$. On each floor branch of \eqref{eq:phi}, the Jacobian of $\phi_c$ is $c^2$, and the inverse branch construction shows that it is an isomorphism modulo boundaries. Since $m_c=c^{-2}m|_{A_c}$, we obtain
\begin{equation}\label{eq:conjugacy}
 (\phi_c)_*m=m_c.
\end{equation}
In fact, Cheung and Quas \cite{CQ} show that $\phi_c \to \mathrm{Id}$ in $m$-probability as $c \uparrow 1$.

For $c_n=1-u/n$, where $n>u$, the return times discrepancy $D_n$ defined in \eqref{definition of D_n} satisfies
\begin{equation}\label{eq:mean-defect}
 \E_m D_n=n(c_n^{-2}-1)\longrightarrow2u,
 \qquad T^{n+D_n(x)}\phi_{c_n}(x)=\phi_{c_n}(T^nx).
\end{equation}

\subsection{Point process statistics}\label{subsec: point process statistics}
For a bounded Borel set $B\subset\R\times(0,\infty)$, define two counting functions
\[
 N_B(H)\coloneq\#(H_{\prim}\cap B)
\]
for $H \in X$ and
\[N_B^\Omega(x)\coloneq N_B(H_x)=\#(\Lambda_x\cap B)\]
for $x \in \Omega$. Both $N_B$ and $N_B^\Omega$ count the number of primitive points in a lattice contained in $B$, but the lattices are sampled from different probability spaces, respectively $(X,\mu)$ and $(\Omega,m)$. Write $\Area(B)$ for the Lebesgue measure of $B$. We will need the following corollary of Theorem 1.2 in \cite{ArtilesMoments}.

\begin{lemma}\label{Bounds on NOmega}
Let $B$ be a bounded Borel set contained in $\R\times(0,\infty)$. For every $c\in(0,1]$ and $x\in\Omega$, we have
\[
N_B(d_cH_x)=N_{d_c^{-1}B}^\Omega(x).
\]
Moreover,
\begin{align}
\E_m\!\left[N_B(d_cH_x)\right] &\leq 2\Area(B), \label{bound on exp nmbo} \\
m\bigl(\{x\in\Omega:N_B(d_cH_x)\geq1\}\bigr) &\leq 2\Area(B). \label{bound on nmbo big}
\end{align}
In particular, taking $c=1$ gives the same bounds for $N_B^\Omega$. The constants are independent of $c$.
\end{lemma}

\begin{proof}
By \cite[Theorem~1.2]{ArtilesMoments},
\begin{equation}\label{albert formula}
\E_m(N_B^\Omega) = \int_{\R}\int_0^\infty \ind_B(x,y)\frac{2\Phi(y)}{y^2}\,dy\,dx,
\end{equation}
where
\[
\Phi(y)\coloneq \sum_{n=1}^{\lfloor y\rfloor}\varphi(n),
\]
and $\varphi$ denotes Euler's totient function.

For $0<y<1$, we have $\Phi(y)=0$. For $y\geq1$, the bound $\varphi(n)\leq n\leq y$ for $1\leq n\leq\lfloor y\rfloor$ gives
\[
\Phi(y) \leq \sum_{n=1}^{\lfloor y\rfloor}y = y\lfloor y\rfloor \leq y^2.
\]
Consequently, \eqref{albert formula} implies
\[
\E_m(N_B^\Omega) \leq 2\int_{\R}\int_0^\infty\ind_B(x,y)\,dy\,dx = 2\Area(B).
\]

The map $d_c$ sends primitive vectors of $H_x$ bijectively onto primitive vectors of $d_cH_x$. Therefore,
\[
N_B(d_cH_x) = N_{d_c^{-1}B}(H_x) = N_{d_c^{-1}B}^\Omega(x).
\]
Since $d_c^{-1}$ preserves area and the positive-height half-plane, applying the preceding first-moment bound to $d_c^{-1}B$ gives
\[
\E_m\!\left[N_B(d_cH_x)\right] = \E_m\!\left[N_{d_c^{-1}B}^\Omega\right] \leq 2\Area(d_c^{-1}B) = 2\Area(B).
\]
This proves \eqref{bound on exp nmbo}.

Finally,
\[
\ind_{\{N_B(d_cH_x)\geq1\}} \leq N_B(d_cH_x).
\]
Integrating with respect to $m$ and applying \eqref{bound on exp nmbo} proves \eqref{bound on nmbo big}.
\end{proof}

\section{Limits of graph measures}\label{sec:graph}

In this section we deduce Theorem~\ref{thm: BCZ is mixing} from Theorem~\ref{thm:defect}. We study subsequential limits of the distributions of the pairs $(x,T^nx)$, where $x$ is distributed according to $m$; this is a routine way to show a system is mixing. What is novel about our approach is that, for each fixed $u>0$, we decompose such a limit according to the return times discrepancy. Then, using the weak mixing of $T$, established in \cite[Theorem~1]{CQ}, we show that every component corresponding to a positive discrepancy is a scalar multiple of $m\otimes m$. Theorem~\ref{thm:defect} then bounds the mass of the remaining component, allowing us to conclude that the limit is $m\otimes m$.

Consider the compact space $\overline\Omega^{\,2}$, extending $m$ by zero outside $\Omega$. For $n\geq1$, define the graph measure on $\overline\Omega^{\,2}$ by
\[
 \Delta_n\coloneq(\mathrm{Id},T^n)_*m.
\]
Thus, for Borel sets $A,B\subset\Omega$,
\[
 \Delta_n(A\times B) = m(A\cap T^{-n}B).
\]
We use $\overset{*}{\rightharpoonup}$ to denote weak$^*$ convergence on $\overline\Omega^{\,2}$. By compactness, every sequence of graph measures has a weak$^*$ convergent subsequence. The standard characterization of mixing by graph measures therefore reduces the proof of Theorem~\ref{thm: BCZ is mixing} to showing that every limit of $\Delta_{n_j}$ with $n_j\to\infty$ is $m\otimes m$; see \cite{deLaRue2006}.

Fix such a limit
\[
 \Delta_{n_j}\overset{*}{\rightharpoonup}\lambda.
\]
Both marginals of $\lambda$ are $m$, so $\lambda$ gives no mass to the boundary. Fix $u>0$, discard the finitely many indices with $n_j\leq u$, and set $\beta_j \coloneq c_{n_j} = 1-u/n_j$.
Define, for $i \in \N_0$,
\[
 G_j(x)\coloneq\bigl(\phi_{\beta_j}(x),T^{n_j}\phi_{\beta_j}(x)\bigr),\qquad \lambda_{j,i}\coloneq(G_j)_*\bigl(\ind_{\{D_{n_j}=i\}}m\bigr).
\]
Thus $\lambda_{j,i}$ is the part of the distribution of $G_j(x)$ coming from points with $D_{n_j}(x)=i$. In particular, its total mass is $m(D_{n_j}=i)$. The positive measures of mass at most one on $\overline\Omega^{\,2}$ are also weak$^*$ compact. Passing to a diagonal subsequence, we may therefore assume that $\lambda_{j,i}\overset{*}{\rightharpoonup}\lambda_i$ for every $i$.
Write
\[
 p_i\coloneq\lambda_i(\overline\Omega^{\,2})
     =\lim_{j\to\infty}m(D_{n_j}=i),
 \qquad
 \mathcal D\coloneq T\times T,
 \qquad S\coloneq\mathrm{Id}\times T,
\]
and let $\pi_1,\pi_2$ be the coordinate projections. We extend $T$ measurably to $\overline\Omega$; the choice on the boundary will not affect any of the limiting measures below.

The following lemma gives two decompositions of $\lambda$ and describes the invariance and marginals of their components. Its proof is rather standard.
\begin{lemma}\label{lem:graph-decompositions}
The measures above satisfy
\begin{align}
 \lambda & =\sum_{i\geq0}\lambda_i
          =\sum_{i\geq0}(S^i)_*\lambda_i,
          \label{eq:graph-decompositions}\\
 \mathcal D_*\lambda_i&=\lambda_i,\qquad
 (\pi_1)_*\lambda_i=(\pi_2)_*\lambda_i=p_i m
 \quad(i\geq0),\nonumber\\
 \sum_{i\geq0}p_i&=1,\qquad
 \sum_{i\geq0}ip_i\leq2u.
 \label{eq:graph-component-masses}
\end{align}
\end{lemma}

\begin{proof}
For a finite signed measure $\nu$, write
$\|\nu\|_{\TV}=\sup_{|f|\leq1}|\int f\,d\nu|$, with the supremum over measurable functions. By \eqref{eq:conjugacy},
\[
 \sum_{i\geq0}\lambda_{j,i} =(\mathrm{Id},T^{n_j})_*m_{\beta_j}.
\]
Since $\|m_{\beta_j}-m\|_{\TV}=2(1-\beta_j^2)\to0$, this sum has weak$^*$ limit $\lambda$. Next, Markov's inequality and \eqref{eq:mean-defect} give, for $M \geq 1$,
\begin{equation}\label{eq:graph-label-tail}
 \sum_{i>M}\lambda_{j,i}(\overline\Omega^{\,2}) =m(D_{n_j}>M) \leq\frac{\E_m D_{n_j}}{M},\qquad
 \E_m D_{n_j}=n_j(\beta_j^{-2}-1) \longrightarrow 2u.
\end{equation}
Passing first to the limit in finite sums and then letting $M\to\infty$ proves $\lambda=\sum_i\lambda_i$ and $\sum_i p_i=1$. Likewise,
\[
 \sum_{i=0}^M ip_i =\lim_{j\to\infty}\sum_{i=0}^M i\,m(D_{n_j}=i) \leq 2u,
\]
which proves the other bound in \eqref{eq:graph-component-masses} upon letting $M\to\infty$.

We next prove invariance of the individual components. Let
\[
 q_j\coloneq m\bigl(\{x:r_{\beta_j}(\phi_{\beta_j}x)>1\}\bigr) \leq \beta_j^{-2}-1\longrightarrow0,
\]
where the inequality follows from \eqref{eq:kac} and \eqref{eq:conjugacy}. The definition of $D_{n_j}$ and the conjugacy give
\[
 D_{n_j}(Tx)-D_{n_j}(x) =r_{\beta_j}(\phi_{\beta_j}T^{n_j}x)-r_{\beta_j}(\phi_{\beta_j}x).
\]
The discrepancy can change only if at least one of the two return times on the right exceeds $1$. Each of these events has measure $q_j$, since $T$ preserves $m$. Hence
\[
 m(D_{n_j}\circ T\neq D_{n_j})\leq2q_j.
\]
Also, $\mathcal D G_j(x)=G_j(Tx)$ whenever $r_{\beta_j}(\phi_{\beta_j}x)=1$.

To compare the measures, put $E_{j,i}\coloneq\{D_{n_j}=i\}$. For a measurable function $f$ with $|f|\leq1$, invariance of $m$ under $T$ gives
\begin{align*}
 \int f\,d(\mathcal D_*\lambda_{j,i}-\lambda_{j,i}) = &\int \ind_{E_{j,i}}(x) \bigl[f(\mathcal D G_j(x))-f(G_j(Tx))\bigr]\,dm(x)\\
 &+\int\bigl[\ind_{E_{j,i}}(x)-\ind_{E_{j,i}}(Tx)\bigr] f(G_j(Tx))\,dm(x).
\end{align*}
The first integral has absolute value at most $2q_j$, since its integrand vanishes when $r_{\beta_j}(\phi_{\beta_j}x)=1$. The second has absolute value at most $m(D_{n_j}\circ T\neq D_{n_j})\leq 2q_j$. Taking the supremum over $f$, we obtain
\begin{equation}\label{eq:graph-almost-invariance}
 \|\mathcal D_*\lambda_{j,i}-\lambda_{j,i}\|_{\TV} \leq 4q_j\longrightarrow0.
\end{equation}

We now pass this invariance to the limit. We have
\[
 (\pi_1)_*\lambda_{j,i}\leq m_{\beta_j}\leq \beta_j^{-2}m,
 \qquad
 (\pi_2)_*\lambda_{j,i} \leq(T^{n_j})_*m_{\beta_j}\leq \beta_j^{-2}m.
\]
Testing against nonnegative continuous functions and passing to the limit shows that both marginals of $\lambda_i$ are dominated by $m$. Each fixed power of $T$ is continuous outside an $m$-null set consisting of the boundary and countably many line segments, so the discontinuity sets of $\mathcal D$ and $S^i$ are $\lambda_i$-null. The continuous mapping theorem (see, e.g., \cite[Theorem~2.7]{Billingsley}) therefore gives
\[
 \mathcal D_*\lambda_{j,i} \overset{*}{\rightharpoonup}\mathcal D_*\lambda_i,
 \qquad
 (S^i)_*\lambda_{j,i} \overset{*}{\rightharpoonup}(S^i)_*\lambda_i.
\]
Here, when $p_i>0$, we apply the theorem after normalizing each measure by its total mass. When $p_i=0$, both conclusions follow directly from convergence of the total masses to zero. Together with \eqref{eq:graph-almost-invariance}, this proves $\mathcal D_*\lambda_i=\lambda_i$. Its marginals are therefore $T$-invariant measures dominated by $m$. Since $T$ is ergodic, each marginal is a scalar multiple of $m$; see \cite[Section~1.3]{deLaRue2006}. Its mass identifies this scalar as $p_i$.

It remains to prove the second decomposition in \eqref{eq:graph-decompositions}. When $D_{n_j}(x)=i$, the return-time identity in \eqref{eq:mean-defect} gives
\[
 S^iG_j(x) =\bigl(\phi_{\beta_j}x,T^{n_j+i}\phi_{\beta_j}x\bigr) =\bigl(\phi_{\beta_j}x,\phi_{\beta_j}T^{n_j}x\bigr).
\]
Summing over $i$, we obtain
\begin{equation}\label{eq:graph-shift-identity}
 \sum_{i\geq0}(S^i)_*\lambda_{j,i} =(\phi_{\beta_j},\phi_{\beta_j}\circ T^{n_j})_*m =(\phi_{\beta_j}\times\phi_{\beta_j})_*\Delta_{n_j}.
\end{equation}
For every $\delta>0$, the fixed marginals of $\Delta_{n_j}$ imply
\begin{align*}
 \Delta_{n_j}\bigl(\{(x,y): \max\{\|\phi_{\beta_j}x-x\|,\|\phi_{\beta_j}y-y\|\}>\delta\}\bigr) \leq 2m\bigl(\{x:\|\phi_{\beta_j}x-x\|>\delta\}\bigr) \longrightarrow 0,
\end{align*}
where $\|\cdot\|$ is the Euclidean norm. Here we used $\phi_c\to\mathrm{Id}$ in $m$-probability from Section~\ref{sec:Preliminaries}. Thus the right-hand side of \eqref{eq:graph-shift-identity} converges weak$^*$ to $\lambda$ by, e.g., \cite[Theorem~3.1]{Billingsley}. The convergence of the individual pushforwards, together with \eqref{eq:graph-label-tail}, now gives $\lambda=\sum_{i\geq0}(S^i)_*\lambda_i$.
\end{proof}
\noindent The next lemma uses weak mixing to identify components corresponding to positive discrepancy.
\begin{lemma}\label{lem:graph-positive-components}
For every $i\geq1$, we have $\lambda_i=p_i(m\otimes m)$. In particular,
\begin{equation}\label{eq:graph-product-domination}
 \lambda\geq(1-p_0)(m\otimes m).
\end{equation}
\end{lemma}

\begin{proof}
To compare the two decompositions of $\lambda$, we sum the intermediate shifts of each component. Set
\[
 \nu\coloneq\sum_{i\geq1}\sum_{r=0}^{i-1}(S^r)_*\lambda_i, \qquad C\coloneq\sum_{i\geq1}ip_i.
\]
By \eqref{eq:graph-component-masses}, $0\leq C\leq2u$. If $C=0$, then, for every $i \geq 1$, we have $p_i=0$, hence $\lambda_i=0$, so the conclusion is immediate. We therefore assume $C>0$.
The series defining $\nu$ converges in total variation, and both marginals of $\nu$ are $Cm$. Therefore,
\begin{align*}
 S_*\nu-\nu =\sum_{i\geq1}\sum_{r=0}^{i-1} \bigl((S^{r+1})_*\lambda_i-(S^r)_*\lambda_i\bigr) = \sum_{i\geq1}\bigl((S^i)_*\lambda_i-\lambda_i\bigr)=0,
\end{align*}
where the last equality follows from \eqref{eq:graph-decompositions}. Thus $\nu/C$ has marginals $m$ and is invariant under $S=\mathrm{Id}\times T$. Ergodicity of $T$ implies that this measure is $m\otimes m$ \cite[Proposition~2.2]{deLaRue2006}, so
\[
 \nu=C(m\otimes m).
\]

For $i\geq1$, the $r=0$ term in the definition of $\nu$ gives $\lambda_i\leq C(m\otimes m)$. Since $T$ is weakly mixing, $\mathcal D=T\times T$ is ergodic with respect to $m\otimes m$. The measure $\lambda_i$ is $\mathcal D$-invariant and dominated by $C(m\otimes m)$, so it is a scalar multiple of $m\otimes m$; see \cite[Section~1.3]{deLaRue2006}. Since its mass is $p_i$, we obtain $\lambda_i=p_i(m\otimes m)$. Summing over $i\geq1$ proves \eqref{eq:graph-product-domination}.
\end{proof}

\noindent We now use Theorem~\ref{thm:defect} to control the remaining component and complete the proof of mixing.
\begin{proof}[Proof of Theorem~\ref{thm: BCZ is mixing}]
Fix a subsequential limit $\lambda$ of the graph measures. For each fixed $u>0$, choose the components as above. Theorem~\ref{thm:defect} gives
\[
 p_0=\lim_{j\to\infty}m(D_{n_j}=0)\leq\frac1{1+2u}.
\]
Lemmas~\ref{lem:graph-decompositions} and \ref{lem:graph-positive-components} therefore imply
\[
 \lambda\geq(1-p_0)(m\otimes m) \geq\frac{2u}{1+2u}\,(m\otimes m).
\]
The further subsequence used to define the components may depend on $u$, but its total graph limit is the same fixed measure $\lambda$. Letting $u\to\infty$ therefore gives $\lambda\geq m\otimes m$. Both measures have mass $1$, so $\lambda=m\otimes m$. The characterization of mixing by graph measures recalled above now proves that $T$ is strongly mixing.
\end{proof}

\section{Statistics of points in thin rectangles}\label{sec: statistics in thin boxes}

The goal of this section is to prove the following proposition, which tells us about the mean and variance of $N_B$, where $B$ is a long, thin rectangle near the line $x = 1$.

\begin{proposition}\label{prop:haar}
Let $1/2\le c<1$ and $h \geq 2$. Set
\[
 B \coloneq (c,1]\times(0,h],\qquad \varepsilon \coloneq 1-c,\qquad V \coloneq \varepsilon h.
\]
Then
\begin{align}
 \E_\mu N_B&=\frac V{\zeta(2)},\label{eq:haar-mean}\\
 \left|\Var_\mu N_B-\frac V{\zeta(2)}\right|
 &\le \frac{2\varepsilon^2h}{c\zeta(2)}(3+\log h). \label{eq:haar-variance}
\end{align}
\end{proposition}
Equation \eqref{eq:haar-mean} follows from the primitive Siegel mean value formula when the test function is $\ind_B$ \cite{Siegel}. We prove \eqref{eq:haar-variance} after some preparation.

We use Fairchild's description \cite[Theorem~2.1]{Fairchild} of Schmidt's pair formula \cite[Section~6]{Schmidt}. Let $\eta$ be the Haar measure on $\mathrm{SL}(2,\R)$ normalized so that the measure of a fundamental domain for the right action of $\mathrm{SL}(2,\Z)$ is $\zeta(2)$. Let $\varphi$ denote Euler's totient function, and let $e_1=(1,0)$ and $e_2=(0,1)$. Fix $B$ as in Proposition~\ref{prop:haar}. Subtracting the diagonal contribution from the second moment and using $B\cap(-B)=\varnothing$ to remove the contribution from pairs $(v,-v)$ gives

\begin{equation}\label{eq: Schmidt formula}
    \E_\mu[N_B(N_B-1)]= \frac{1}{\zeta(2)}\sum_{k\in \Z\setminus \{0\}}\varphi(|k|)\int_{\mathrm{SL}(2,\R)}\ind_B(ge_1)\ind_B(g(e_1+ke_2))\,d\eta(g).
\end{equation}

We now proceed to compute these integrals. Let $\det:\R^2\times\R^2\to\R$ be the determinant function. Let $\operatorname{Leb}$ be Lebesgue measure on $\R^2\times\R^2$ restricted to $B\times B$, and let $I : \R \to [0,\infty)$ be the Radon--Nikodym derivative of its pushforward under the determinant map, so that $\det\nolimits_{*}\operatorname{Leb}=I(t)\,dt$. Thus, for every Borel set $E\subset\R$,
\begin{equation}\label{eq:pushforward}
 \operatorname{Leb}\bigl(\{(v,w)\in B\times B:\det(v,w)\in E\}\bigr)
 =\int_E I(t)\,dt.
\end{equation}

The following Lemma gives us a formula for $I(t)$ which we will use in subsequent computations.

\begin{lemma}
\label{lem:determinant-density}
Fix notation as in Proposition~\ref{prop:haar}. The pushforward under the determinant map of Lebesgue measure restricted to $B\times B$ has a density $I:\mathbb R\to[0,\infty)$ given by
\begin{equation}\label{eq:explicit I(t)}
I(t)= \int_c^1\int_c^1\int_0^h \frac1r \mathbf 1_{\{0<(t+sY)/r\le h\}} \,dY\,ds\,dr.
\end{equation}
\end{lemma}

\begin{proof}
Let $\psi:\mathbb R\to[0,\infty)$ be a nonnegative Borel function. Write $v=(r,Y)$ and $w=(s,W)$. Since $\det(v,w)=rW-sY$, Tonelli's theorem gives
\begin{equation}\label{eq:initial}
\int_{B\times B}\psi(\det(v,w))\,dv\,dw = \int_c^1\int_c^1\int_0^h\int_0^h \psi(rW-sY)\,dW\,dY\,ds\,dr.
\end{equation}
For fixed $r,s,Y$, the substitution $t=rW-sY$ yields
\[
\int_0^h\psi(rW-sY)\,dW = \int_{\mathbb R} \psi(t)\frac1r \mathbf 1_{\{0<(t+sY)/r\le h\}}\,dt.
\]
Substituting this identity into \eqref{eq:initial} and applying Tonelli's theorem again, we obtain
\[
\begin{aligned}
\int_{B\times B}\psi(\det(v,w))\,dv\,dw & = \int_{\mathbb R}\psi(t) \left[ \int_c^1\int_c^1\int_0^h \frac1r \mathbf 1_{\{0<(t+sY)/r\le h\}} \,dY\,ds\,dr \right]dt\\
& = \int_{\mathbb R}\psi(t)I(t)\,dt.
\end{aligned}
\]
Thus \eqref{eq:explicit I(t)} defines a representative of the Radon--Nikodym derivative in \eqref{eq:pushforward}.
\end{proof}

The following lemma shows how to compute the integrals in \eqref{eq: Schmidt formula}.
\begin{lemma}\label{Lem: Schmidt and I(t)}
   Fix notation as in Proposition~\ref{prop:haar}. Let $I$ be as in \eqref{eq:explicit I(t)}. For every $k\in\Z\setminus\{0\}$,
    \begin{equation*}
    \int_{\mathrm{SL}(2,\R)}\ind_B(ge_1)\ind_B(g(e_1+ke_2))\,d\eta(g)= \frac{I(k)}{|k|}.
    \end{equation*}
\end{lemma}

\begin{proof}
    
Write the columns of $g\in\mathrm{SL}(2,\mathbb R)$ as
\[
ge_1=v=(r,Y),
\qquad
ge_2=(0,r^{-1})+\tau v.
\]
On the region where $v\in B$, we have $r>0$, and the Haar measure $\eta$ used in \eqref{eq: Schmidt formula} is
\[
d\eta=dr\,dY\,d\tau.
\]
Set $w=g(e_1+ke_2)$, and let $s$ be the first coordinate of $w$.
Then
\[
s=r(1+k\tau), \qquad \left|\frac{d\tau}{ds}\right|=\frac{1}{|k|r}.
\]
Since $\det(v,w)=k$, the second coordinate of $w$ is
\[
W=\frac{k+sY}{r}.
\]
Consequently,
\begin{align*}
\int_{\mathrm{SL}(2,\R)} \mathbf 1_B(ge_1)\, \mathbf 1_B\bigl(g(e_1+ke_2)\bigr)\,d\eta(g) = \frac{1}{|k|} \int_c^1\int_c^1\int_0^h \frac{1}{r} \mathbf 1_{\{0<(k+sY)/r\le h\}} \,dY\,ds\,dr = \frac{I(k)}{|k|}.
\end{align*}
This completes the proof.
\end{proof}

The following lemma collects useful facts about $I$.

\begin{lemma}\label{lem: I(t) is symmetric in t} Fix notation as in Proposition~\ref{prop:haar}, and let $I$ be as in \eqref{eq:explicit I(t)}. Then $I$ is non-increasing and absolutely continuous on $[0,\infty)$, is zero outside $(-h,h)$, is even, and satisfies
\begin{equation}\label{eq:I integral}
    \int_{0}^h I(t) dt = \frac{V^2}{2} \, .
\end{equation}
\end{lemma}
\begin{proof}
We first prove pointwise symmetry using \eqref{eq:explicit I(t)}. Fix $t\in\mathbb R$. For fixed $r,s\in[c,1]$, make the substitution
\[
W=\frac{-t+sY}{r}, \qquad Y=\frac{t+rW}{s}.
\]
The condition $0<Y\leq h$ becomes $0<(t+rW)/s\leq h$, and the indicator in the defining integral restricts $W$ to $(0,h]$. Consequently,
\[
\begin{aligned}
I(-t) & = \int_c^1\int_c^1\int_0^h
\frac1r \mathbf 1_{\{0<(-t+sY)/r\leq h\}} \,dY\,ds\,dr\\
& = \int_c^1\int_c^1\int_0^h \frac1s \mathbf 1_{\{0<(t+rW)/s\leq h\}} \,dW\,ds\,dr\\
& = I(t),
\end{aligned}
\]
where the last equality follows by exchanging $r$ and $s$. Thus $I$ is even at every point.

For $t\geq0$, integrating first with respect to $Y$ in
\eqref{eq:explicit I(t)} gives
\[
I(t)= \int_c^1\int_c^1\frac1r \min\left\{h,\frac{\max\{0,rh-t\}}{s}\right\}\,ds\,dr.
\]
The integrand is non-increasing
in $t$, so $I$ is non-increasing on $[0,\infty)$. Moreover, $I(t)=0$ for $t\geq h$, since $r\leq1$. By evenness, $I(t)=0$ whenever $|t|\geq h$. For $t,t'\in \R$, we have that
\[
|I(t)-I(t')| \leq |t-t'|\int_c^1\int_c^1\frac1{rs}\,ds\,dr = \log^2(1/c)\,|t-t'|.
\]
Thus $I$ is Lipschitz, and hence absolutely continuous, on $[0,\infty)$. Since $I$ is even, it is also continuous on all of $\mathbb R$.

Finally, taking $E=\mathbb R$ in \eqref{eq:pushforward} gives
\[
\int_{\mathbb R}I(t)\,dt =\operatorname{Leb}(B\times B) =V^2.
\]
Using evenness and the fact that $I(t)=0$ for $|t|\geq h$, we obtain
\[
\int_0^h I(t)\,dt=\frac{V^2}{2},
\]
which proves \eqref{eq:I integral}.
\end{proof}

Now, using \eqref{eq: Schmidt formula}, together with Lemma~\ref{Lem: Schmidt and I(t)} and the symmetry of $I(t)$ in Lemma~\ref{lem: I(t) is symmetric in t}, we have the following identity for the second factorial moment of $N_B$:
\begin{equation}\label{eq:pair}
 \E_\mu[N_B(N_B-1)]
 =\frac1{\zeta(2)}\sum_{k\in\Z\setminus\{0\}}
 \frac{\varphi(|k|)}{|k|}I(k)
 =\frac2{\zeta(2)}\sum_{k\ge1}\frac{\varphi(k)}k I(k).
\end{equation}

It will be convenient to approximate the right-hand side of \eqref{eq:pair}. Let $\mu_{\mathrm{Mob}}$ denote the M\"obius function, and set 
\[
F(t) \coloneq \sum_{1\leq k\leq t}\frac{\varphi(k)}{k}.
\]
The standard number-theoretic identities
\[
\frac{\varphi(k)}{k} = \sum_{d\mid k}\frac{\mu_{\mathrm{Mob}}(d)}{d}
\qquad\text{and}\qquad
\sum_{d\geq1}\frac{\mu_{\mathrm{Mob}}(d)}{d^2} =\frac{1}{\zeta(2)}
\]
imply, for $t \in [0,h]$ and $h \geq 2$, 
\begin{equation}\label{eq:totientbound}
 \left|F(t)-\frac{t}{\zeta(2)}\right|\le3+\log h
 .
\end{equation}

\begin{lemma}\label{lem:totient-weighted-I}
Fix notation as in Proposition~\ref{prop:haar}, and let $I$ be as in \eqref{eq:explicit I(t)}. Then
\begin{equation*}
\left| \sum_{k\geq1}\frac{\varphi(k)}{k}I(k) -\frac{1}{\zeta(2)}\int_0^h I(t)\,dt \right| \leq (3+\log h)I(0).
\end{equation*}
\end{lemma}

\begin{proof}
On the one hand, since $I(h)=0$ and $I$ is absolutely continuous on $(0,\infty)$, we have
\begin{equation*}
I(k)=\int_k^h -I'(t)\,dt
\end{equation*}
for $k \in [1,h]$. Since $I$ is non-increasing on $(0,\infty)$, it then follows that
\begin{equation*}
\sum_{k\geq1}\frac{\varphi(k)}{k} I(k) = \int_0^h F(t)(-I'(t))\,dt.
\end{equation*}
On the other hand, $I(h) = 0$ and integration by parts gives
\begin{equation*}
\int_0^h I(t)\,dt = \int_0^h t(-I'(t))\,dt.
\end{equation*}
Combining these identities with the bound \eqref{eq:totientbound} implies
\begin{equation*}
\begin{aligned}
\left| \sum_{k\geq1}\frac{\varphi(k)}{k}I(k) -\frac{1}{\zeta(2)}\int_0^h I(t)\,dt \right|
& = \left| \int_0^h \left(F(t)-\frac{t}{\zeta(2)}\right)(-I'(t))\,dt \right| \\
& \leq(3+\log h)\int_0^h-I'(t)\,dt =(3+\log h)I(0). \qedhere
\end{aligned}
\end{equation*}
\end{proof}

\begin{proof}[Proof of Proposition~\ref{prop:haar}]
It remains only to show \eqref{eq:haar-variance}. Lemma~\ref{lem:totient-weighted-I}, along with \eqref{eq:I integral} and \eqref{eq:pair}, gives
\begin{equation}\label{combining things}
 \left|\E_\mu[N_B(N_B-1)]-\frac{V^2}{\zeta(2)^2}\right| \le\frac{2(3+\log h)}{\zeta(2)}I(0).
\end{equation}
At $t=0$, \eqref{eq:explicit I(t)} gives
\begin{equation}\label{eq:Izero}
I(0)=h\int_c^1\int_c^1\frac{ds\,dr}{\max(r,s)} =2h(1-c+c\log c)\le\frac{\varepsilon^2h}{c}.
\end{equation}
Since $\Var_\mu N_B=\E_\mu[N_B(N_B-1)]+\E_\mu N_B-(\E_\mu N_B)^2$, combining \eqref{combining things} and \eqref{eq:Izero} proves \eqref{eq:haar-variance}.
\end{proof}

\section{Statistics of return times discrepancy}\label{sec: D_n Statistics}
Fix $u>0$, and, for $n>u$, recall $c_n=1-\frac{u}{n}$. We use the following notation:
\[
 B_n \coloneq (c_n,1]\times(0,\bar\rho n],\qquad N_n(H) \coloneq N_{B_n}(H) \qquad \text{ for } H\in X.
\]
The rectangle $B_n$ has width $u/n$ and fixed area $u\bar\rho$. On $(\Omega,m)$, set
\[ M_n(x) \coloneq N_{B_n}(d_{c_n}H_x)=N_{B_n}(H_{\phi_{c_n}x}).
\] 
The law of $M_n$ samples from $\Omega$ with the measure $m$, whereas the law of $N_n$ samples the full space $X$ with the measure $\mu$. 

The next proposition compares the law of $D_n$ to the law of $N_n$ when $n$ is large. Its proof appears near the end of this section, after which we prove Theorem~\ref{thm:defect}.
\begin{proposition}\label{prop:comparison of D_n and N_n}
For every fixed $u>0$,
\begin{equation*}
 \lim_{n\rightarrow\infty}\sup_{J\subset\N_0}\bigl|m(D_n\in J)-\mu(N_n\in J)\bigr|=0.
\end{equation*}
\end{proposition}
Taking $J=\{0\}$ will transfer the Haar estimate for empty rectangles to the desired bound on $m(D_n=0)$. We first compare $D_n$ with the section count $M_n$, and then pass from the section to the suspension. This second comparison introduces the roof weight $R$, which we remove using the almost invariance of the discrepancy events.  Setting $J=\{1\}$ allows us to recover the fact that the BCZ map is weakly mixing. See Remark \ref{Rem: Recovering CQ}.

\begin{lemma}\label{lem:defect-rectangle-comparison}
    For a fixed $u>0$,
    \begin{equation*}
        \lim_{n\rightarrow \infty}m(D_n\neq M_n)=0.
    \end{equation*}
\end{lemma}

\begin{proof}
Fix $0<\delta<\bar\rho$, recall \eqref{definition of t_n}, and define
\begin{equation*}
G_{n,\delta} \coloneq \left\{x\in\Omega:
(\bar\rho-\delta)n\leq t_n(x)\leq(\bar\rho+\delta)n
\right\}.
\end{equation*}
The ergodic theorem gives $t_n(x)/n\to\bar\rho$ for $m$-almost every $x$. Consequently, 
\begin{equation} \label{gndcomp small} \lim_{n\to\infty} m(G_{n,\delta}^{c}) = 0.
\end{equation}

The linear map $d_{c_n}$ multiplies slopes by $c_n^{-2}$ and sends the horizontal cutoff $1$ to $c_n$. Thus the $n$th return of $H_x$ to the section corresponds to the $n$th return of $H_{\phi_{c_n}x}=d_{c_n}H_x$ to $A_{c_n}$. By the definition of $D_n$, this is its $(n+D_n(x))$-th return to $\Omega$, so
\begin{equation*}
 t_{n+D_n(x)}(\phi_{c_n}x)=c_n^{-2}t_n(x).
\end{equation*}
The intervening returns to $\Omega\setminus A_{c_n}$ correspond exactly to primitive vectors $(v_1,v_2)\in d_{c_n}H_x$ with $c_n<v_1 \leq 1$ and $0<v_2/v_1 \leq c_n^{-2}t_n(x)$. Therefore,
\begin{equation*}
D_n(x) = \#\left\{ (v_1,v_2)\in(d_{c_n}H_x)_{\mathrm{prim}}: c_n<v_1 \leq1,\quad 0<v_2\leq\frac{v_1t_n(x)}{c_n^2}
\right\},
\end{equation*}
whereas
\begin{equation*}
M_n(x) = \#\left\{ (v_1,v_2)\in(d_{c_n}H_x)_{\mathrm{prim}}: c_n<v_1\leq1,\quad 0<v_2\leq\bar\rho n
\right\}.
\end{equation*}
For $x\in G_{n,\delta}$ and $c_n<v_1\leq1$,
\begin{equation*}
(\bar\rho-\delta)n
\leq\frac{(\bar\rho-\delta)n}{c_n}
\leq\frac{v_1t_n(x)}{c_n^2}
\leq\frac{(\bar\rho+\delta)n}{c_n^2}.
\end{equation*}

Therefore, if $x\in G_{n,\delta}$, then every vector counted by exactly one of $D_n(x)$ and $M_n(x)$ belongs to the rectangle
\[
C_{n,\delta}\coloneq (c_n,1]\times \left( (\bar\rho-\delta)n,\frac{(\bar\rho+\delta)n}{c_n^2} \right].
\]
In particular,
\begin{equation*}
\{D_n\neq M_n\} \subset G_{n,\delta}^{c} \cup \left\{ x:N_{C_{n,\delta}}(d_{c_n}H_x)\geq1 \right\}.
\end{equation*}
Lemma~\ref{Bounds on NOmega} then implies that
\begin{equation*}
m\left\{ x:N_{C_{n,\delta}}(d_{c_n}H_x)\geq1 \right\} \leq 2\Area(C_{n,\delta}).
\end{equation*}
Since $n(1-c_n)=u$, the area of the strip is
\begin{equation*}
\Area(C_{n,\delta}) = n(1-c_n) \left( \frac{\bar\rho+\delta}{c_n^2} -(\bar\rho-\delta) \right) = \frac{u}{c_n^2} \left[ \bar\rho(1-c_n^2)+\delta(1+c_n^2) \right].
\end{equation*}
Combining the above observations, we conclude
\begin{equation*}
m(D_n\neq M_n) \leq m(G_{n,\delta}^{c}) + \frac{2u}{c_n^2} \left[ \bar\rho(1-c_n^2)+\delta(1+c_n^2) \right].
\end{equation*}
Letting $n\to\infty$ and using $c_n\to1$ and \eqref{gndcomp small} yields
\begin{equation*}
\limsup_{n\to\infty}m(D_n\neq M_n)\leq4u\delta.
\end{equation*}
Since $\delta>0$ can be arbitrarily small, the result follows.
\end{proof}

The next lemma shows that, for fixed $u>0$, the events determined by $D_n$ become asymptotically invariant under $T$, uniformly over all subsets of possible discrepancy values.

\begin{lemma}\label{lem: one return discrepancy}
Fix $u > 0$. For $n > u$ and $J\subset \N_0$, define $E_{n,J}\coloneq \{x\in\Omega:D_n(x)\in J\}$. Then
\[
\lim_{n\to\infty} \sup_{J\subset\mathbb{N}_0}m(E_{n,J}\triangle T^{-1}E_{n,J}) = 0.
\]
\end{lemma}
\begin{proof}
Define $q_n \coloneq m_{c_n}\{r_{c_n}>1\}$. Kac's lemma gives
\begin{equation*}
q_n \leq \int_{A_{c_n}}(r_{c_n}-1)\,dm_{c_n}
=c_n^{-2}-1 \longrightarrow 0.
\end{equation*}

Notice
\begin{equation*}
D_n(Tx)-D_n(x) = r_{c_n}(\phi_{c_n}T^n x)-r_{c_n}(\phi_{c_n}x).
\end{equation*}
Therefore,
\begin{equation*}
\{D_n\circ T\neq D_n\} \subset \{r_{c_n}\circ\phi_{c_n}>1\} \cup T^{-n}\{r_{c_n}\circ\phi_{c_n}>1\}.
\end{equation*}
Since $(\phi_{c_n})_*m=m_{c_n}$ and $T$ preserves $m$, each event on the right has measure $q_n$. Consequently,
\begin{equation*}
m(D_n\circ T\neq D_n)\leq2q_n.
\end{equation*}

Finally, for every $J\subset\mathbb{N}_0$, we have $E_{n,J}\triangle T^{-1}E_{n,J} \subset \{D_n \circ T \neq D_n\}$, hence
\begin{equation*}
\sup_{J\subset\mathbb{N}_0} m(E_{n,J}\triangle T^{-1}E_{n,J}) \leq 2q_n \longrightarrow 0. \qedhere
\end{equation*}
\end{proof}

The following lemma is a standard consequence of ergodicity. Besides that, it does not use any specific information about the BCZ map.
\begin{lemma}
\label{lem:almost-invariant-weight}
Suppose that a family of measurable sets $E_{n,J} \subset \Omega$ satisfies
\begin{equation*}
\lim_{n\to\infty} \sup_{J\subset\N_0} m(E_{n,J}\triangle T^{-1}E_{n,J}) = 0.
\end{equation*}
Then, for every $g\in L^1(m)$,
\begin{equation*}
\lim_{n\to\infty} \sup_{J\subset \N_0} \left| \int_\Omega g\mathbf{1}_{E_{n,J}}\,dm -m(E_{n,J})\int_\Omega g\,dm \right|
= 0.
\end{equation*}
\end{lemma}

\begin{proof}
Write
\[
\theta_n\coloneq \sup_{J\subset\N_0}m(E_{n,J}\triangle T^{-1}E_{n,J}),
\]
so that $\theta_n \to 0$. First suppose that $g$ is bounded. Fix $n,J$, and write $z=\mathbf{1}_{E_{n,J}}$. Telescoping and invariance of $m$ give, for $r \in \N$,
\[
\|z\circ T^r-z\|_{L^1(m)}\leq r\theta_n.
\]
Since $T$ preserves $m$, it follows that
\[
\left|\int_\Omega gz\,dm-\int_\Omega(g\circ T^r)z\,dm\right| = \left|\int_\Omega(g\circ T^r)(z\circ T^r-z)\,dm\right| \leq r\theta_n\|g\|_\infty.
\]
For $K \geq 1$, define
\[
\sigma_Kg \coloneq \frac1K\sum_{r=0}^{K-1}g\circ T^r.
\]
Averaging the preceding inequality and using $0\leq z \leq 1$ yields
\[
\sup_{J\subset\N_0} \left| \int_\Omega g\mathbf{1}_{E_{n,J}}\,dm -m(E_{n,J})\int_\Omega g\,dm \right| \leq \frac{K-1}{2}\theta_n\|g\|_\infty + \left\|\sigma_Kg-\int_\Omega g\,dm\right\|_{L^1(m)}.
\]
For fixed $K$, the first term tends to zero as $n\to\infty$. Ergodicity and boundedness of $g$ imply that $\sigma_Kg\to\int_\Omega g\,dm$ in $L^1(m)$ as $K\to\infty$. Taking these limits in this order proves the assertion for bounded $g$. The extension to $g \in L^1(m)$ is routine.
\end{proof}

Taking $g=R$ in the previous lemma with $E_{n,J}$ as in Lemma~\ref{lem: one return discrepancy} yields
\begin{equation}\label{eq:roofweight}
\lim_{n\to\infty} \sup_{J\subset\N_0} \left| \frac{1}{\bar\rho} \int_\Omega R(x)\ind_{\{D_n(x)\in J\}}\,dm(x) -m(D_n\in J) \right| = 0.
\end{equation}

\begin{proof}[Proof of Proposition~\ref{prop:comparison of D_n and N_n}]
Fix $L>0$. The identity
\begin{equation*}
d_{c_n}h_sH_x = h_{s/c_n^2}d_{c_n}H_x
\end{equation*}
shows that, for $0\leq s\leq L$, a vector $(v_1,v_2)$ in the scaled lattice with $c_n<v_1\leq1$ moves downward by $sv_1/c_n^2\leq L/c_n^2$. Thus the count in $B_n$ can change only if $d_{c_n}H_x$ contains a primitive vector in one of the strips
\begin{equation*}
\begin{aligned}
C_{n,L}^{-} & \coloneq (c_n,1]\times\left(0,\frac{L}{c_n^2}\right],\\
C_{n,L}^{+} & \coloneq (c_n,1]\times \left(\bar\rho n,\bar\rho n+\frac{L}{c_n^2}\right].
\end{aligned}
\end{equation*}
These strips account for vectors leaving through the bottom edge or entering through the top edge of $B_n$, respectively. Their total area is at most $2L(1-c_n)/c_n^2$. 

Define
\begin{equation*}
F_{n,L} \coloneq \left\{ x\in\Omega: N_{C_{n,L}^{-}}(d_{c_n}H_x) + N_{C_{n,L}^{+}}(d_{c_n}H_x) \geq 1 \right\}.
\end{equation*}
Lemma~\ref{Bounds on NOmega} gives
\begin{equation*}
m(F_{n,L}) \leq\frac{4L(1-c_n)}{c_n^2}.
\end{equation*}
For $x\notin F_{n,L}$, we thus have, for $s \in [0,L]$,
\begin{equation*}
N_n(d_{c_n}h_sH_x)=M_n(x).
\end{equation*}

For every $J\subset\mathbb{N}_0$, Haar invariance under $d_{c_n}$ and the suspension formula imply
\begin{equation*}
\mu(N_n\in J) = \int_X \mathbf{1}_{\{N_n(d_{c_n}H)\in J\}}\,d\mu(H) = \frac{1}{\bar\rho} \int_\Omega\int_0^{R(x)} \mathbf{1}_{\{N_n(d_{c_n}h_sH_x)\in J\}} \,ds\,dm(x).
\end{equation*}
If $x\notin F_{n,L}$ satisfies $R(x)\leq L$ and $D_n(x)=M_n(x)$, then for $s \in [0,R(x))$,
\[
N_n(d_{c_n}h_sH_x)=D_n(x).
\]
Thus, on $\{R\leq L\}$, the integration of the absolute difference of $\mathbf{1}_{\{D_n(x)\in J\}}$ and $\mathbf{1}_{\{N_n(d_{c_n}h_sH_x)\in J\}}$  contributes at most $L$ times the measure of $F_{n,L}\cup\{D_n\neq M_n\}$. On $\{R>L\}$, the same difference between the indicator functions contributes at most $R(x)$ on each fiber. Integrating over $\{R>L\}$ gives a contribution at most $\int_{\{R>L\}}R(x)\,dm(x)$. Hence
\begin{equation}\label{eq:suspension-comparison}
\begin{aligned}
&\sup_{J\subset\mathbb{N}_0} \left| \mu(N_n\in J) -\frac{1}{\bar\rho} \int_\Omega R(x)\mathbf{1}_{\{D_n(x)\in J\}}\,dm(x) \right|\\
& \qquad\leq \frac{L}{\bar\rho} \left[ m(D_n\neq M_n) +\frac{4L(1-c_n)}{c_n^2} \right] + \frac{1}{\bar\rho}\int_{\{R>L\}}R\,dm.
\end{aligned}
\end{equation}

By Lemma~\ref{lem:defect-rectangle-comparison}, $m(D_n\neq M_n) \to 0$. Also, $c_n \to 1$. Thus, for fixed $L$, the term in square brackets tends to zero as $n\to\infty$. Combining \eqref{eq:suspension-comparison} with \eqref{eq:roofweight} gives
\begin{equation*}
\limsup_{n\to\infty} \sup_{J\subset\mathbb{N}_0} \left| \mu(N_n\in J)-m(D_n\in J) \right| \leq \frac{1}{\bar\rho}\int_{\{R>L\}}R\,dm.
\end{equation*}
Since $R\in L^1(m)$, the right-hand side tends to zero as $L\to\infty$. This proves Proposition~\ref{prop:comparison of D_n and N_n}.
\end{proof}

\begin{proof}[Proof of Theorem \ref{thm:defect}]
Fix $u>0$. By Proposition~\ref{prop:haar}, for sufficiently large $n$, the Haar rectangle count $N_n$ satisfies
\begin{equation}\label{eq:zero-event-moments}
\mathbb{E}_\mu N_n=2u,
\qquad
\operatorname{Var}_\mu(N_n)\leq2u+e_n,
\qquad
e_n=\frac{4u^2}{c_n n} \bigl(3+\log(\bar\rho n)\bigr)\longrightarrow0.
\end{equation}

For every nonnegative random variable $N$ with positive mean and finite second moment, Cauchy--Schwarz gives

\begin{equation*}
\mathbb{P}(N=0) \leq \frac{\operatorname{Var}(N)}
{(\mathbb{E}N)^2+\operatorname{Var}(N)}.
\end{equation*}
For a fixed positive mean, the last expression is increasing as a function of the variance. Applying \eqref{eq:zero-event-moments} therefore yields
\begin{equation*}
\mu(N_n=0) \leq \frac{2u+e_n}{4u^2+2u+e_n}.
\end{equation*}

We deduce that
\begin{equation*}
m(D_n=0) \leq \frac{2u+e_n}{4u^2+2u+e_n} + \left|m(D_n=0)-\mu(N_n=0)\right|.
\end{equation*}
Letting $n\to\infty$ and applying Proposition \ref{prop:comparison of D_n and N_n} proves
\begin{equation*}
\limsup_{n\to\infty}m(D_n=0) \leq \frac{2u}{4u^2+2u} = \frac{1}{1+2u}. \qedhere
\end{equation*}
\end{proof}

\begin{remark}\label{Rem: Recovering CQ}
The same estimates recover weak mixing of the BCZ map by the Cheung and Quas criterion \cite[Lemma~2]{CQ}. For $0<u<1/2$, the inequality $\mathbf{1}_{\{N=1\}}\geq N-N(N-1)$ for $N \in \N_0$ and \eqref{eq:zero-event-moments} give
\[
 \mu(N_n=1)\geq 2u-4u^2-e_n.
\]
Proposition~\ref{prop:comparison of D_n and N_n} therefore implies
\[
 \liminf_{n\to\infty}m(D_n=1)\geq 2u-4u^2.
\]
Choosing $u=1/4$, one gets a uniform bound of $1/4$ for $\liminf_{n\to\infty}m(D_n=1)$. Moreover, the conjugacy and the definition of $D_n$ give
\[
 m\bigl(\{y\in A_{c_n}: \sum_{j=0}^{n-1}r_{c_n}(T_{c_n}^j y)=n+1\}\Bigr )=c_n^2m(D_n=1).
\]
Thus the event that the $n$th induced return occurs at ordinary time $n+1$ has measure bounded away from zero for all sufficiently large $n$. The other hypotheses of the criterion are of course satisfied already in \cite{CQ}. Applying the criterion then proves weak mixing. 
\end{remark}

\noindent \textbf{Acknowledgments.} A version of Theorem~\ref{thm:defect} was formulated through interaction with ChatGPT 5.6 Sol. We thank Yitwah Cheung, Sohail Farhangi, and Tianru Zhu for helpful discussions.


\begin{thebibliography}{99}

\bibitem{ArtilesMoments}
A.~Artiles,
\emph{Moments of the cross-sectional Siegel--Veech transform},
arXiv:2609.36463 [math.DS], 2026.
\url{https://arxiv.org/abs/2609.36463}.

\bibitem{AC}
J.~S.~Athreya and Y.~Cheung,
\emph{A Poincar\'e section for the horocycle flow on the space of lattices},
Int. Math. Res. Not. IMRN \textbf{2014} (2014), no.~10, 2643--2690.

\bibitem{Banff}
J.~S.~Athreya, F.~P.~Boca, J.~Chaika, Y.~Cheung,
N.~Chevallier, S.~Leli\`evre, J.~Tseng, and A.~Quas,
\emph{Banff workshop report: Generalized Gauss maps and Farey statistics},
BIRS workshop report, September 16--23, 2012.
\url{https://archytas2.birs.ca/workshops//2012/12frg172/report12frg172.pdf}.

\bibitem{Billingsley}
P.~Billingsley,
\emph{Convergence of Probability Measures},
2nd ed., Wiley, 1999.

\bibitem{BCZHall}
F.~P.~Boca, C.~Cobeli, and A.~Zaharescu,
\emph{A conjecture of R.~R.~Hall on Farey points},
J. Reine Angew. Math. \textbf{535} (2001), 207--236.
\href{https://doi.org/10.1515/crll.2001.049}
{doi:10.1515/crll.2001.049}.

\bibitem{BZ}
F.~P.~Boca and A.~Zaharescu,
\emph{Farey fractions and two-dimensional tori},
in \emph{Noncommutative Geometry and Number Theory},
Aspects Math., vol.~E37,
Vieweg, Wiesbaden, 2006, pp.~57--77.

\bibitem{CQ}
Y.~Cheung and A.~Quas,
\emph{BCZ map is weakly mixing},
arXiv:2403.14976, 2024.
\url{https://arxiv.org/abs/2403.14976}.

\bibitem{deLaRue2006}
T.~de la Rue,
\emph{An introduction to joinings in ergodic theory},
Discrete Contin. Dyn. Syst. \textbf{15} (2006), no.~1, 121--142.
\href{https://doi.org/10.3934/dcds.2006.15.121}
{doi:10.3934/dcds.2006.15.121}.

\bibitem{Fairchild}
S.~K.~Fairchild,
\emph{A higher moment formula for the Siegel--Veech transform
over quotients by Hecke triangle groups},
Groups Geom. Dyn. \textbf{15} (2021), no.~1, 57--81.
\href{https://doi.org/10.4171/GGD/591}
{doi:10.4171/GGD/591}.

\bibitem{Hall1994}
R.~R.~Hall,
\emph{On consecutive Farey arcs II},
Acta Arith. \textbf{66} (1994), no.~1, 1--9.

\bibitem{Schmidt}
W.~M.~Schmidt,
\emph{A metrical theorem in geometry of numbers},
Trans. Amer. Math. Soc. \textbf{95} (1960), 516--529.

\bibitem{Siegel}
C.~L.~Siegel,
\emph{A mean value theorem in geometry of numbers},
Ann. of Math. (2) \textbf{46} (1945), 340--347.

\bibitem{Walters}
P.~Walters,
\emph{An Introduction to Ergodic Theory},
Graduate Texts in Mathematics, vol.~79,
Springer-Verlag, New York, 1982.

\end{thebibliography}
\end{document}